\documentclass[-ejs]{imsart}

\RequirePackage[OT1]{fontenc}
\RequirePackage{amsthm, amsmath, amsfonts, amssymb,epsf, graphicx}
\RequirePackage[square]{natbib}
\RequirePackage[colorlinks,citecolor=blue,urlcolor=blue]{hyperref}
\RequirePackage{hypernat}

\makeatletter \@addtoreset{equation}{section} \makeatother

\newcommand{\RR}{{\mathbb  R}}

\newcommand{\ZZ}{{\mathbb Z}}

\newtheorem{thm}{Theorem}[section]
\newtheorem{lem}[thm]{Lemma}
\newtheorem{prop}[thm]{Proposition}
\newtheorem{cor}[thm]{Corollary}

\advance \topmargin by -\headheight
\advance \topmargin by -\headsep
\advance \topmargin .2in
\begin{document}

\begin{frontmatter}
\title{Global Rates of Convergence of the MLE for Multivariate Interval Censoring} 
\runtitle{Multivariate interval censoring global rates}

\begin{aug}
\author{\fnms{Fuchang} \snm{Gao}\thanksref{t1}\ead[label=e1]{fuchang@uidaho.edu}}
\and
\author{\fnms{Jon A.} \snm{Wellner}\thanksref{t2}\ead[label=e2]{jaw@stat.washington.edu}}
\ead[label=u1,url]{http://www.stat.washington.edu/jaw/}

\thankstext{t1}{Supported in part by a grant from the Simons Foundation (\#246211)}
\thankstext{t2}{Supported in part by NSF Grant DMS-1104832 and NI-AID grant 2R01 AI291968-04} 
\runauthor{Gao and Wellner}


\address{Department of Mathematics \\University of Idaho\\Moscow, Idaho 83844-1103\\
\printead{e1}}

\address{Department of Statistics, Box 354322\\University of Washington\\Seattle, WA  98195-4322\\
\printead{e2}\\
\printead{u1}}
\end{aug}

\begin{abstract}
We establish global rates of convergence of the Maximum Likelihood Estimator (MLE) of 
a multivariate distribution function on $\RR^d$ in the case of (one type of) ``interval censored''
data.  The main finding is that the rate of convergence of the MLE in the Hellinger metric
is no worse than $n^{-1/3} (\log n)^{\gamma}$ for $\gamma = (5d - 4)/6$. 
\end{abstract}



\end{frontmatter}


\bigskip

\bigskip

\section{Introduction and overview}
\label{sec:intro}

Our main goal in this paper is to study global rates of convergence of 
the Maximum Likelihood Estimator (MLE) in one simple model for multivariate 
interval-censored data.  In section 3 we will show that under some reasonable
conditions the MLE converges in a Hellinger metric to the true distribution 
function on $\RR^d$ at a rate no worse than $n^{-1/3} (\log n)^{\gamma_d}$ for $\gamma_d = (5d - 4)/6$ 
for all $d\ge 2$.  Thus the rate of convergence is only worse than the known rate of $n^{-1/3}$ for 
the case $d=1$ by a  factor involving a power of $\log n$ growing linearly with the dimension.
These new rate results rely heavily on recent bracketing entropy bounds for $d-$dimensional 
distribution functions obtained by \cite{Gao:12}.

We begin in Section~\ref{sec:IntCensOnR} with a review of interval censoring problems and known results in the case
$d=1$.  We introduce the multivariate interval censoring model of interest here in 
Section~\ref{sec:MultIntervalCensoring}, and 
obtain a rate of convergence for this model for $d\ge 2$ in Theorem~\ref{NewMultHellingerRateThm}.
Most of the proofs are given in Section~\ref{sec:proofs}, with the exception being 
a key corollary of \cite{Gao:12}, the statement and proof of which are given in 
the Appendix (Section~\ref{sec:Appendix}).
Finally, in Section~\ref{sec:OtherModelsFurtherProblems} 
we introduce several related models and further problems.

\section{Interval Censoring (or Current Status Data) on $\RR$}
\label{sec:IntCensOnR}

Let $Y\sim F_0$ on $\RR^+$,  and let $T \sim G_0$ on $\RR^+$ be independent of $Y$.  Suppose that we 
observe $X_1, \ldots , X_n$ i.i.d. as $X = (\Delta , T)$ where $\Delta = 1_{[Y \le T]}$.  
Here $Y$ is often the time until some event of interest and $T$ is an observation time.  
The goal is to estimate $F_0$ nonparametrically based on observation of the $X_i$'s.

To calculate the likelihood, we first calculate the distribution of $X$ 
for a general distribution function $F$:  note that the conditional distribution of $\Delta$ conditional on $T$ is
Bernoulli:  
$$
( \Delta | T ) \sim \mbox{Bernoulli} (p(T))
$$
where $p(T) = F(T)$.   If $G_0$ has density $g_0$ with respect to some measure $\mu$ on $\RR^+$, then
$X = (\Delta , T)$ has density 
$$
p_{F,g_0} (\delta, t) = F(t)^{\delta} (1-F(t))^{1-\delta} g_0 (t),   \ \ \ \delta \in \{ 0,1 \},  \ \ t \in \RR^+,
$$
with respect to the dominating measure (counting measure on $\{ 0,1 \}) \times \mu$.  

The nonparametric Maximum Likelihood Estimator (MLE) $\hat{F}_n$ of $F_0$ in this interval 
censoring model was first obtained by 
\cite{MR0073895}.  It is simply described as follows:  let $T_{(1)} \le \cdots \le T_{(n)}$ denote 
the order statistics corresponding to $T_1, \ldots , T_n$ and let 
$\Delta_{(1)}, \ldots , \Delta_{(n)}$ denote the corresponding $\Delta$'s.  
Then the part of the log-likelihood of $X_1 ,\ldots , X_n$ depending on $F$ 
is given by 
\begin{eqnarray}
l_n (F) & = & \sum_{i=1}^n \{ \Delta_{(i)} \log F(T_{(i)}) + (1-\Delta_{(i)} ) \log (1-F(T_{(i)})) \} \nonumber\\
& \equiv & \sum_{i=1}^n \{ \Delta_{(i)} \log F_{i} + (1-\Delta_{(i)} ) \log (1-F_{i}) \} \label{LogLikelihood}
\end{eqnarray}
where 
\begin{eqnarray}
0 \le F_1 \le \cdots \le F_n \le 1 . 
\label{MonConstr}
\end{eqnarray} 
It turns out that the maximizer $\hat{F}_n$ of (\ref{LogLikelihood}) subject to (\ref{MonConstr})
can be described as follows:
let $H^*$ be the (greatest) convex minorant of the points $\{ (i, \sum_{j\le i} \Delta_{(j)}): \ i \in \{ 1, \ldots , n \} \}$:
\begin{eqnarray*}
H^* (t) = \sup \left \{ \begin{array}{l} H(t) : \ \ H(i) \le \sum_{j\le i} \Delta_{(j)}\ \ \mbox{for each} \ \ 0 \le i \le n \\
                                                         \ \ H(0) = 0, \ \mbox{and} \ \ H \ \ \mbox{is convex} 
                                 \end{array} \right \} .
\end{eqnarray*}
Let $\hat{F}_i$ denote the left-derivative of $H^*$ at $T_{(i)}$.  Then 
$(\hat{F}_1, \ldots , \hat{F}_n)$ is the unique vector maximizing (\ref{LogLikelihood}) subject to 
(\ref{MonConstr}), and we therefore take the MLE $\hat{F}_n$ of $F$ to be
\begin{eqnarray*}
\hat{F}_n (t) = \sum_{i=0}^n \hat{F}_i 1_{[T_{(i)}, T_{(i+1)})} (t) 
\end{eqnarray*}
with the conventions $T_{(0)} \equiv 0$ and $T_{(n+1)} \equiv \infty$.  See 
\cite{MR0073895} or \cite{MR1180321}, pages 38-43, for details.  

\cite{Groeneboom:87} initiated the study of $\hat{F}_n$ and proved the following limiting 
distribution result at a fixed point $t_0$.  

\begin{thm}  (Groeneboom, 1987).
\label{GroeneboomPointwiseLimitDistributionRone}
Consider the current status model on $\RR^+$.
Suppose that $0 < F_0(t_0) , G_0(t_0) < 1$ and suppose that 
$F$ and $G$ are differentiable at $t_0$ with strictly positive derivatives 
$f_0(t_0)$ and $g_0(t_0)$ respectively.  Then
\begin{eqnarray*}
n^{1/3} ( \hat{F}_n (t_0) - F_0(t_0) ) \rightarrow_d c(F_0,G_0)  \ZZ
\end{eqnarray*}
where 
\begin{eqnarray*}
c(F_0,G_0) = 2 \left ( \frac{F_0(t_0) (1-F_0(t_0))f_0(t_0)}{2g_0 (t_0)} \right )^{1/3} 
\end{eqnarray*}
and 
\begin{eqnarray*}
\ZZ = \mbox{argmin} \{ W(t) + t^2 \}
\end{eqnarray*}
where $W$ is a standard two-sided Brownian motion starting from $0$.
\end{thm}
\medskip

The distribution of $\ZZ$ has been studied in detail by 
\cite{MR981568}  
and computed by 
\cite{MR1939706}.  
\cite{BalabdaWell:12} show that the density $f_{\ZZ}$ of $\ZZ$ is log-concave.

\cite{MR1212164} (see also \cite{MR1739079}) obtained the following 
global rate result for $p_{\hat{F}_n}$.  
Recall that the Hellinger distance $h(p,q)$ between two densities with respect
to a dominating measure $\mu$ is given by 
$$
h^2 (p,q) = \frac{1}{2} \int \{ \sqrt{p} - \sqrt{q} \}^2 d \mu .
$$

\begin{prop} (van de Geer, 1993)
\label{VdGHellingerRateRone}
$ h( p_{\hat{F}_n} , p_{F_0}) = O_p (n^{-1/3})$.
\end{prop}

\par\noindent
Now for any distribution functions $F$ and $F_0$ the (squared) Hellinger distance $h^2 (p_F , p_{F_0})$ 
for the current status model is given by 
\begin{eqnarray}
h^2 (p_F , p_{F_0}) 
&= & \frac{1}{2} \left \{ \int ( \sqrt{F} - \sqrt{F_0} )^2 dG_0 + \int ( \sqrt{1-F} - \sqrt{1-F_0} )^2 d G_0  \right \}  \nonumber \\
& = &  \frac{1}{2} \int \frac{\{ ( \sqrt{F} - \sqrt{F_0} )(\sqrt{F} + \sqrt{F_0}) \}^2}{ (\sqrt{F} + \sqrt{F_0})^2} dG_0 \nonumber \\
&& \ \ + \  \frac{1}{2} \int \frac{\{ ( \sqrt{1-F} - \sqrt{1-F_0} )(\sqrt{1-F} + \sqrt{1-F_0}) \}^2}{ (\sqrt{1-F} + \sqrt{1-F_0})^2} dG_0
               \nonumber \\
& \ge & \frac{1}{8} \int (F - F_0 )^2 d G_0 +  \frac{1}{8} \int ((1-F)- (1-F_0) )^2 d G_0 \nonumber \\
& = & \frac{1}{4}  \int (F - F_0 )^2 d G_0 ,  \label{L2-Hellinger-Inequal}
\end{eqnarray}
and hence 
Proposition~\ref{VdGHellingerRateRone} yields
\begin{eqnarray}
\int_0^\infty ( \hat{F}_n (z) - F_0(z) )^2 d G_0 (z) = O_p (n^{-2/3}) ,
\end{eqnarray}
or $\| \hat{F}_n - F_0 \|_{L_2 (G_0)} = O_p (n^{-1/3})$.
\medskip

For generalizations of these and other asymptotic results 
for the current status model to more complicated interval censoring schemes
for real-valued random variables $Y$, see e.g. 
\cite{MR1180321},  
\cite{MR1212164},  
\cite{MR1600884},  
\cite{MR1739079},  
\cite{MR1774042},  
and \cite{MR2418648, MR2418649}.  

Our main focus in this paper, however, concerns one simple generalization of the 
interval censoring model for $\RR$ introduced above to interval censoring in $\RR^d$.  
We now turn to this generalization.

\section{Multivariate interval censoring:  multivariate current status data}
\label{sec:MultIntervalCensoring}

Let $\underline{Y} = (Y_1, \ldots , Y_d) \sim F_0$ on $\RR^{+d} \equiv [0,\infty)^d$,
and let $\underline{T} = (T_1, \ldots , T_d) \sim G_0 $ on $\RR^{+d}$ be independent of $\underline{Y}$.
We assume that $G_0$ has density $g_0$ with respect to some dominating measure $\mu$ on $\RR^d$.
Suppose we observe $\underline{X}_1, \ldots , \underline{X}_n$ i.i.d. as 
$\underline{X} = ( \underline{\Delta}, \underline{T} )$ where $\underline{\Delta} = (\Delta_1 , \ldots , \Delta_d)$ 
is given by $\Delta_j = 1_{[Y_j \le T_j ]}$, $j =1, \ldots , d$.  
Equivalently, with a slight abuse of notation,  $\underline{X} = (\underline{\Gamma}, \underline{T} )$ 
where $\underline{\Gamma} = ( \Gamma_1 , \ldots , \Gamma_{2^d} )$
is a vector of length $2^d$ consisting of $0$'s and $1$'s and with at most one $1$ 
which indicates into which of the $2^d$ orthants of $\RR^{+d}$ determined by 
$\underline{T}$ the random vector $\underline{Y}$ 
belongs.  More explicitly, define $K\equiv 1+\sum_{j=1}^d (1-\Delta_j) 2^{j-1}$.  Then set 
$\Gamma_k \equiv 1\{ k =K  \}$ for $k=1,\ldots , 2^d$, so that $\Gamma_K = 1$ and $\Gamma_{l} = 0$ for 
$l \in \{ 1, \ldots , 2^d \} \setminus \{ K \}$.  Much as for univariate current status data, 
$\underline{Y}$ represents a vector of times to events, $\underline{T}$ is a vector of observation times, 
and the goal is nonparametric estimation of the joint distribution function $F_0$ of $\underline{Y}$ based on observation 
of the $\underline{X}_i$'s.  
See \cite{MR1891046},  
\cite{MR2416109},   
\cite{MR2736679},   
and \cite{MR2818097}  
for examples of settings in which data of this type arises.

To calculate the likelihood, we first calculate the distribution of $\underline{X}$
for a general distribution function $F$:  note that the conditional distribution of $\underline{\Gamma}$ 
conditional on $\underline{T}$ is Multinomial:
$$
( \underline{\Gamma} | \underline{T} ) \sim \mbox{Mult}_{2^d} ( 1, \underline{p}(\underline{T}; F))
$$
where $\underline{p}(\underline{T};F) = (p_1 (\underline{T};F) , \ldots , p_{2^d} (\underline{T};F))$ and the probabilities 
$p_j (\underline{t};F)$, $j = 1, \ldots , 2^d$, $\underline{t} \in \RR^{+d}$ are determined 
by the $F$ measures of the corresponding sets.  
Then our model ${\cal P}$ for multivariate current status data 
 is the collection of all densities with respect to the dominating measure 
$(\mbox{counting measure on} \ \{ 0,1\}^{2^d}) \times \mu$ given by 
$$
\prod_{j=1}^{2^d} p_j (\underline{t};F)^{\gamma_j} g_0 (\underline{t}) 
$$
for some distribution function $F$ on $\RR^{+d}$ where 
$\underline{t} \in \RR^{+d}$ and $\gamma_j \in \{ 0, 1\}$ with $\sum_{j=1}^{2^d} \gamma_j = 1$. 

Now the part of the log-likelihood that depends on $F$ is given by
\begin{eqnarray*}
l_n (F) = \sum_{i=1}^n \sum_{j=1}^{2^d} \Gamma_{i,j} \log p_j (\underline{T}_i ; F),
\end{eqnarray*}
and again the MLE  $\hat{F}_n$ of the true distribution function $F_0$ is given by
\begin{eqnarray}
\hat{F}_n = \mbox{argmax} \{ l_n (F) : \ F \ \mbox{is a distribution function on} \ \RR^{+d} \}.
\label{MultivariateIntCensMLE}
\end{eqnarray}

For example, when $d=2$, we can write 
$\Gamma_1 = \Delta_1 \Delta_2$, 
$\Gamma_2 = (1-\Delta_1)\Delta_2$,
$\Gamma_3 = \Delta_1 (1-\Delta_2)$,  and 
$\Gamma_4 = (1-\Delta_1)(1-\Delta_2)$, and then 
\begin{eqnarray*}
&& p_1 (\underline{T};F) = F(T_1, T_2), \\
&& p_2 (\underline{T};F) = F(\infty, T_2) - F(T_1, T_2) , \\
&& p_3 (\underline{T};F) = F(T_1, \infty) - F(T_1 , T_2) , \\
&& p_4 (\underline{T};F) = 1- F(T_1 , \infty) - F(\infty, T_2 ) + F(T_1, T_2) .
\end{eqnarray*}
Thus 
$$
P_{F} ( \underline{\Gamma} = \underline{\gamma} | \underline{T}) 
       = \prod_{j=1}^4 p_j (\underline{T}; F)^{\gamma_j} , \ \ \mbox{for} 
\ \ \underline{\gamma} = (\gamma_1 , \gamma_2 ,\gamma_3 , \gamma_4), \ \ 
\gamma_j \in \{ 0,1 \}, \ \sum_{j=1}^4 \gamma_j = 1 .
$$
Note that 
\begin{eqnarray}
p_j (\underline{t}; F) = \int_{[0, \infty)^2} 1_{C_j (\underline{t})} (\underline{y}) dF(\underline{y}), \qquad j = 1, \ldots , 4
\label{StructureOfTheCellProbabilities}
\end{eqnarray}
where 
\begin{eqnarray*}
&& C_1 (\underline{t}) = [0,t_1]\times [0,t_2], \\
&& C_2 (\underline{t}) =  [0,t_1]\times (t_2, \infty) ,\\
&& C_3 (\underline{t}) = (t_1, \infty) \times [0,t_2] ,\\
&& C_4 (\underline{t}) = (t_1 , \infty) \times (t_2, \infty) .
\end{eqnarray*}

Characterizations and computation of the MLE (\ref{MultivariateIntCensMLE}), mostly 
for the case $d=2$ have
been treated in 
\cite{Song/PHD},    
\cite{MR1964427},  
and 
\cite{MR2160818, MR2708977}.  
Consistency of the MLE for more general interval censoring models 
has been established by \cite{MR2236498}.
For an interesting application see   
\cite{Betensky-Fink:99}.    
This example and other examples of multivariate interval censored data
are treated in \cite{MR2287318}   
and and \cite{MR2489672}.   
For a comparison of the MLE with alternative estimators in the case $d=2$, see 
\cite{Groeneboom:12}.  

An analogue of Groeneboom's Theorem~\ref{GroeneboomPointwiseLimitDistributionRone} 
has not been established in the multivariate case.
\cite{Song/PHD} established an asymptotic minimax lower bound for pointwise convergence when $d=2$:   if
$F_0$ and $G_0$ have positive continuous densities at $\underline{t}_0$, then no estimator has a local minimax
rate for estimation of $F_0 (\underline{t}_0)$ faster than $n^{-1/3}$.  
By making use of additional smoothness hypotheses,
\cite{Groeneboom:12} has constructed estimators which achieve the pointwise  $n^{-1/3}$
rate, but it is not yet known if the MLE achieves this.

Our main goal here is to prove the following theorem concerning the global rate of convergence of the 
MLE $\hat{F}_n$.
 
\begin{thm}
\label{NewMultHellingerRateThm} 
Consider the multivariate current status model.
Suppose that $F_0$ has $\mbox{supp} (F_0) \subset [0,M]^d$ and that $F_0$ has density $f_0$ 
which satisfies
\begin{eqnarray}
c_1^{-1} \le f_0 (\underline{y}) \le c_1 \ \ \mbox{for all} \ \ \underline{y} \in [0,M]^d
\label{ConditionOne}
\end{eqnarray}
where $0< c_1 < \infty$.  Suppose that $G_0$ has density $g_0$ which 
satisfies
\begin{eqnarray}
c_2^{-1} \le g_0 (\underline{y}) \le c_2 \ \ \mbox{for all} \ \ \underline{y} \in [0,M]^d.
\label{ConditionTwo}
\end{eqnarray} 
Then the MLE $\widehat{p}_n \equiv  p_{\widehat{F}_n}$ of $p_0\equiv p_{F_0}$ satisfies 
\begin{eqnarray*}
h( \widehat{p}_n , p_0 ) = O_p \left ( \frac{ (\log n)^{\gamma}}{n^{1/3}} \right ) 
\end{eqnarray*}
for  $\gamma \equiv \gamma_d \equiv (5d-4)/6$.
\end{thm}

Since the inequality (\ref{L2-Hellinger-Inequal}) continues to hold in $\RR^d$ for $d\ge 2$ 
(with $1/4$ replaced by $1/8$ on the right side),  we obtain the 
following corollary:
  
\begin{cor}
\label{L2-rateCorollaryMultCase}
Under the conditions of Theorem~\ref{NewMultHellingerRateThm} it follows that 
\begin{eqnarray*}
\int_{\RR^{+d}} ( \widehat{F}_n (z) - F_0 (z) )^2 d G_0 (z) = O_p ( n^{-2/3} (\log n)^{\beta} ) 
\end{eqnarray*}
for  $\beta \equiv \beta_d = 2 \gamma_d = (5d- 4)/3$.
\end{cor}

\section{Proofs}  
\label{sec:proofs}

Here we give the proof of Theorem~\ref{NewMultHellingerRateThm}.  
The main tool is a method developed by 
\cite{MR1739079}.  We will use the following lemma in combination with 
Theorem 7.6 of \cite{MR1739079} or Theorem 3.4.1 of \cite{MR1385671} (Section 3.4.2, pages 330-331).
Without loss of generality we can take $M=1$ where $M$ is the upper bound of the support of $F$ 
(see Theorem~\ref{NewMultHellingerRateThm}).  

Let ${\cal P}$ be a collection of probability densities $p$ on a sample space ${\cal X}$ with respect to 
a dominating measure $\mu$.  Define
\begin{eqnarray}
&& {\cal G}^{(conv)} \equiv \left \{ \frac{2 p}{p+p_0} : \ p \in {\cal P} \right  \} ,
      \label{CalGDefn} \\
&& \sigma( \delta ) \equiv \sup \{ \sigma \ge 0 : \ \ \int_{\{ p_0 \le \sigma \}} p_0 d \mu \le \delta^2 \} 
\qquad \mbox{for} \ \ \delta > 0 ,  
      \label{SigmaSubDeltaDefn}\\
&& {\cal G}_{\sigma}^{(conv)} \equiv \left  \{ \frac{2 p}{p+p_0}  1_{[p_0 > \sigma]} : \ p \in {\cal P} \right \} , 
\qquad \mbox{for} \ \ \sigma > 0 .
     \label{CalGSubSigmaDefn}
\end{eqnarray}
The following general result relating the bracketing entropies $\log N_{[\,]} (\cdot , {\cal G}^{(conv)} , L_2 (P_0))$,
$\log N_{[\,]} (\cdot , {\cal G}_{\sigma(\epsilon)}^{(conv)} , L_2 (P_0))$,
$\log N_{[\,]} (\cdot , {\cal P}_ , L_2 (Q_{\sigma(\epsilon)} ))$,
and
$\log N_{[\,]} (\cdot , {\cal P}_ , L_2 (\tilde{Q}_{\sigma(\epsilon)} ))$
is due to \cite{MR1739079}.

\begin{lem} (van de Geer, 2000)
\label{VdGBasicLemma}
For every $\epsilon > 0$
\begin{eqnarray}
\log N_{[\, ]} ( 3 \epsilon , {\cal G}^{(conv)} , L_2 (P_0) )
& \le & \log N_{[\, ]} ( \epsilon , {\cal G}_{\sigma(\epsilon)}^{(conv)} , L_2 (P_0)) 
              \label{VDG_First_Inequality} \\
& \le & \log N_{[\, ]} ( \epsilon/2 , {\cal P}, L_2 ( Q_{\sigma(\epsilon)} ) ) 
               \label{VDG_Second_Inequality} \\
& = & \log N_{[\, ]} \left ( \frac{\epsilon/2}{\sqrt{Q_{\sigma(\epsilon)} ({\cal X} )}} , 
              {\cal P}, L_2 (\tilde{Q}_{\sigma(\epsilon)} ) \right ) 
                \label{VDG_Third_Equality} 
\end{eqnarray}
where $dQ_{\sigma} \equiv p_0^{-1} 1_{[p_0 > \sigma]} d \mu$ and 
$\tilde{Q}_{\sigma} \equiv Q_{\sigma} / Q_{\sigma} ({\cal X} )$.
\end{lem}

\begin{proof}
We first show that (\ref{VDG_First_Inequality}) holds.  Suppose that 
$\{ [g_{L,j} , g_{U,j} ] , \ j = 1, \ldots , m \}$ are $\epsilon$-brackets with respect to $L_2(P_0)$ 
for ${\cal G}_{\sigma (\epsilon)}^{(conv)}$ with 
$$  
{\cal G}_{\sigma (\epsilon)}^{(conv)}  \subset \bigcup_{j=1}^m  [g_{L,j} , g_{U,j} ] , \qquad 
m = N_{[\, ]} ( \epsilon , {\cal G}_{\sigma(\epsilon)}^{(conv)} , L_2 (P_0)) . 
$$
Then for $g \in {\cal G}^{(conv)}$, let $g_{\sigma} \equiv g 1_{[p_0 > \sigma]}$ be the corresponding 
element of ${\cal G}_{\sigma (\epsilon)}^{(conv)}$.  Suppose that 
$g_{\sigma} \in  [g_{L,j} , g_{U,j} ]$ for some $j \in \{ 1, \ldots , m \}$.  Then 
\begin{eqnarray*}
g = g 1_{[p_0 \le \sigma]} + g_{\sigma} 
\ \left \{ \begin{array}{l}  \le g 1_{[p_0 \le \sigma ]} + g_{U,j} \equiv \tilde{g}_{U,j} \\
                                      \ge 0 + g_{L,j} \equiv \tilde{g}_{L,j} ,
            \end{array} \right .
\end{eqnarray*}
where, by the triangle inequality, $ 0 \le g \le 2 $ for all $g \in {\cal G}^{(conv)}$, 
and the definition of $\sigma(\epsilon)$,
it follows that 
\begin{eqnarray*}
\big \| \tilde{g}_{U,j} - \tilde{g}_{L,j} \big \|_{P_0,2} \le \big \|{g}_{U,j} - {g}_{L,j} \big \|_{P_0,2}  + 2 \epsilon 
\le 3 \epsilon .
\end{eqnarray*}
Thus $\{ [ \tilde{g}_{L,j} , \tilde{g}_{U,j}] : \ j \in \{ 1, \ldots , m\} \}$ is a collection of 
$3\epsilon-$brackets for ${\cal G}^{(conv)}$ with respect to $L_2 (P_0)$ 
and hence (\ref{VDG_First_Inequality}) holds.

Now we show that (\ref{VDG_Second_Inequality}) holds.
Suppose that $ \{ [ p_{L,j}, p_{U,j} ] : \ j = 1, \ldots , m \}$
is a set of $\epsilon/2-$brackets with respect to $L_2 (Q_{\sigma} )$ for ${\cal P}$ with 
$$
{\cal P} \subset \bigcup_{j=1}^m  [ p_{L,j}, p_{U,j} ]  \qquad \mbox{and} \qquad 
m = N_{[\, ]} ( \epsilon/2 , {\cal P}, L_2 ( Q_{\sigma(\epsilon)} ) ) .
$$
Suppose $p \in [ p_{L,j}, p_{U,j} ] $ for some $j$.  Then, since 
\begin{eqnarray*}
\frac{2p}{p+p_0} 1_{[p_0 > \sigma]} \ \ \left \{ 
\begin{array}{l} \le \frac{2p_{U,j}}{p_{U,j} + p_0 } 1_{[p_0 > \sigma]} \equiv g_{U,j} , \\
                          \ \    \\
                         \ge \frac{2p_{L,j}}{p_{U,j} + p_0} 1_{[p_0 > \sigma ]} \equiv g_{L,j}  
\end{array} \right .
\end{eqnarray*}
where
\begin{eqnarray*}
\lefteqn{| g_{U,j} - g_{L,j} |  } \\
& = & \bigg | \frac{2p_{U,j}}{p_{U,j} + p_0 } 1_{[p_0 > \sigma]} - \frac{2p_{L,j}}{p_{U,j} + p_0 } 1_{[p_0 > \sigma]} \bigg | \\
& = &  \frac{2(p_{U,j} - p_{L,j})}{p_{L,j} + p_0 } 1_{[p_0 > \sigma]} 
 \le  \frac{2 | p_{U,j} - p_{L,j} |}{p_0 } 1_{[p_0 > \sigma]} . 
\end{eqnarray*}
Thus
\begin{eqnarray*}
\big \| g_{U,j} - g_{L,j} \|_{P_0 , 2} \le 2 \big \| p_{U,j} - p_{L,j} \big \|_{Q_{\sigma} , 2} \le \epsilon ,
\end{eqnarray*}
and hence $\{ [ g_{L,j}, g_{U,j} ] : \ j = 1, \ldots , m \}$ is a set of $\epsilon$-brackets with respect to 
$L_2 (P_0)$ for ${\cal G}_{\sigma}^{(conv)}$.   This shows that (\ref{VDG_Second_Inequality}) holds.

It remains only to show that (\ref{VDG_Third_Equality}) holds.  But this is easy since
$\| g \|_{Q_{\sigma},2}^2 = \| g \|_{\tilde{Q}_\sigma,2}^2 \cdot Q_{\sigma} ({\cal X} )$.

This lemma is based on
\cite{MR1739079}, pages 101 and 103.
Note that our constants differ slightly from those of van de Geer.
\end{proof}

\begin{lem}
\label{BoundsForFZeroF}
Suppose that $F_0$ has density $f_0$ 
which satisfies, for some $0 < c_1 < \infty$,
\begin{eqnarray}
\frac{1}{c_1} \le f_0 (\underline{y} ) \le c_1 \qquad 
\mbox{for all} \ \ 
\underline{y} \in [0,1]^d .
\label{BoundsForFZeroDensity}
\end{eqnarray}
Then $p_0$ (which we can identify with the vector $p_0 (\cdot , F_0)$) satisfies
\begin{eqnarray*}
&& p_{0,1} (\underline{t}; F_0) \left \{ \begin{array}{l} \le c_1 \prod_{j=1}^d t_j \\ \ge c_1^{-1} \prod_{j=1}^d t_j , \end{array} 
\right . \qquad \mbox{for all} \ \ \underline{t}\in [0,1]^d , \\
&& \vdots \\
&& p_{0,2^d} (\underline{t}; F_0) \left \{ 
   \begin{array}{l} \le c_1 \prod_{j=1}^d (1-t_j) \\ \ge c_1^{-1} \prod_{j=1}^d (1-t_j ), \end{array} 
\right . \qquad \mbox{for all} \ \ \underline{t} \in [0,1]^d .
\end{eqnarray*}
\end{lem}
\begin{proof} 
This follows immediately from the general $d$ version of 
(\ref{StructureOfTheCellProbabilities}) and the assumption on $f_0$.
\end{proof}

These inequalities can also be written in the following compact form:  For $k=1+ \sum_{j=1}^d (1-\delta_j )2^{j-1}$
with $\delta_j \in \{ 0,1\}$, 
\begin{eqnarray*}
p_{0,k} (\underline{t}; F_0)  \left \{ \begin{array}{l} \le c_1 \prod_{j=1}^d t_j^{\delta_j} (1-t_j)^{1-\delta_j} \\
                                                             \ge c_1^{-1} \prod_{j=1}^d t_j^{\delta_j} (1-t_j)^{1-\delta_j} ,
                                     \end{array} \right . \qquad \mbox{for all} \ \ t \in [0,1]^d .
\end{eqnarray*}                       
  
\begin{lem} 
\label{BoundsForSmallValuesDensity}
Suppose that the assumption of Lemma~\ref{BoundsForFZeroF} 
holds.
Suppose, moreover, that 
$G_0$ has density $g_0$ 
which satisfies 
\begin{eqnarray}
\frac{1}{c_2} \le g_0 (\underline{y} ) \le c_2 \qquad 
\mbox{for all} \ \ 
\underline{y} \in [0,1]^d.
\label{BoundsForGZeroDensity} 
\end{eqnarray}
Then
\begin{eqnarray*}
\int_{[p_0 \le \sigma ]} p_0 d \mu \le 2^d (c_1 c_2)^2  \sigma .
\end{eqnarray*}
Furthermore,  with 
$\sigma (\delta ) \equiv  \delta^2 / (2^d (c_1 c_2)^2)$
we have 
$$
\int_{[p_0 \le \sigma( \delta )]} p_0 d \mu \le \delta^2 .
$$
\end{lem}

\begin{proof}
The first inequality follows easily from Lemma~\ref{BoundsForFZeroF}:  note that 
\begin{eqnarray*}
\int_{[p_0 \le \sigma]} p_0 d\mu 
& = & \sum_{k=1}^{2^d} \int_{[p_k (\underline{t}, F_0) \le \sigma]} p_k (\underline{t} , F_0 ) g_0 (\underline{t}) \underline{dt} \\
& \le & 2^d \int_{[F_0 (\underline{t})g_0 (\underline{t}) \le \sigma]} F_0 (\underline{t} )  g_0 (\underline{t} ) \underline{dt} \\
& \le & 2^d c_1 c_2 \int_{[ c_1^{-1} c_2^{-1}  \prod_{j=1}^d t_j \le \sigma]} \prod_{j=1}^d t_j \, \underline{dt}  \le 2^d (c_1 c_2)^2 \sigma .
\end{eqnarray*}
The second inequality follows from the first inequality of the lemma.
\end{proof}
  
\begin{lem}
\label{MassOfQsigma}
If the hypotheses of Lemmas~\ref{BoundsForFZeroF}  and~\ref{BoundsForSmallValuesDensity}  
hold, then the measure $Q_{\sigma}$
defined by $dQ_{\sigma} \equiv (1/p_0) 1 \{ p_0 > \sigma\}d \mu $ 
has total mass $Q_{\sigma} ({\cal X})$
given by
\begin{eqnarray}
\int d Q_{\sigma} 
& = & \int_{\{ p_0  > \sigma\}} \frac{1}{p_0} d \mu
            \nonumber  \\
& = & \sum_{j=1}^{2^d} \int_{\{\underline{t} : \ p_{0,j} (\underline{t}) g_0(\underline{t}) > \sigma \}} 
               \frac{1}{ p_{0,j} (\underline{t}) g_0(\underline{t}) }  \underline{dt} \nonumber \\
& \le & 2^d \int_{\{\underline{t} \in [0,1]^d: \  \prod_{j=1}^d t_j > \sigma/(c_1 c_2) \} }
               \frac{ c_1c_2}{ \prod_{j=1}^d t_j }  \underline{dt}
                \label{FirstCrackAnalyticExpressionTotalMass} \\
& = & \frac{2^d c_1 c_2}{d!} ( \log (c_1 c_2 / \sigma ))^d  .
                \label{AsympApproxTotalMass}
\end{eqnarray}
\end{lem}

\begin{proof}
This follows from Lemma~\ref{BoundsForFZeroF}, 
followed by an explicit calculation. 
In particular, the equality in (\ref{AsympApproxTotalMass}) follows from 
\begin{eqnarray*}
\int_{[\prod_{j=1}^d t_j > b]} \frac{1}{\prod_{j=1}^d t_j}  \underline{dt} 
& = & \int_{[ \sum_1^d x_j \le \log (1/b)]}  \underline{dx} \ \  \mbox{by the change of variables} \ t_j = e^{-x_j},\\
& = &\frac{1}{d!} \left ( \log (1/b) \right )^d \ \  \ \ \mbox{for} \ \ 0 < b \le 1
\end{eqnarray*}
where the second equality follows by induction:  it holds easily for $d=1$ (and $d=2$); and then an easy calculation
shows that it holds for $d$ if it holds for $d-1$.
\end{proof}
 
\begin{lem}
\label{CruxEntropyBound}
If the hypotheses of Lemmas~\ref{BoundsForFZeroF} and ~\ref{BoundsForSmallValuesDensity} hold,
and $d\ge2$, then 
\begin{eqnarray*}
\log N_{[\, ]} ( \epsilon , {\cal G}^{(conv)} , L_2 (P_0)) \le K \frac{ \left [ \log (1/\epsilon)\right ]^{5d/2-2}}{\epsilon} 
\end{eqnarray*}
for all $0 < \epsilon < \mbox{some}\  \epsilon_0 $ and some constant $K< \infty$. 
\end{lem}
  
\begin{proof}
This follows by combining the results of Lemmas~\ref{BoundsForSmallValuesDensity} and 
\ref{MassOfQsigma}  with Lemma~\ref{VdGBasicLemma}, and 
then using Corollary~\ref{CorollariesOfGaosThm}  
of the bracketing entropy bound of \cite{Gao:12} and stated here as Theorem~\ref{GaosThm}.  
Here is the explicit calculation:
\begin{eqnarray*}
\lefteqn{\log N_{[\, ]} ( 6 \epsilon , {\cal G}^{(conv)} , L_2 (P) )} \\
& \le & \log N_{[\, ]} \left ( \frac{\epsilon}{\sqrt{Q_{\sigma(\epsilon)} ({\cal X} )}}, 
              {\cal P}, L_2 (\tilde{Q}_{\sigma(\epsilon)} ) \right )  \qquad \mbox{by Lemma~\ref{VdGBasicLemma}}   \\  
& \le & \log N_{[\, ]} \left ( \frac{\epsilon}{\sqrt{\frac{2^d c_1c_2}{d!}  \left [ \log ((c_1c_2)^3\cdot 2^d /(\epsilon^2 ) \right ]^d }}, {\cal P} , 
                L_2 (\tilde{Q}_{\sigma(\epsilon)} ) \right ) 
                \qquad \mbox{by \ Lemmas~\ref{BoundsForSmallValuesDensity}~and~\ref{MassOfQsigma}} \\
& \le & \log N_{[\, ]} \left ( \frac{V \epsilon}{\left [ \log (1/\epsilon ) \right ]^{d/2}  }, {\cal P} , 
                L_2 (\tilde{Q}_{\sigma(\epsilon)} ) \right ) \qquad \mbox{for} \ \ V = V_d (c_1,c_2) \\
& \le & K \frac{\left [ \log(1/\epsilon) \right ]^{d/2}}{V \epsilon} 
               \left [ \log \left ( \frac{ (\log (1/\epsilon))^{d/2}}{V\epsilon} \right ) \right ]^{2(d-1)}  
                  \qquad \mbox{by Corollary~\ref{CorollariesOfGaosThm}(b)} \\
& \le & \tilde{K} \frac{\left [\log (1/\epsilon) \right ]^{5d/2 -2}}{\epsilon} 
\end{eqnarray*}
for $\epsilon$ sufficiently small. 
\end{proof} 
 
\begin{proof} (Theorem~\ref{NewMultHellingerRateThm})
This follows from Lemma~\ref{CruxEntropyBound} and Theorem 7.6 of 
\cite{MR1739079} or Theorem 3.4.1 of \cite{MR1385671} together with the 
arguments given in Section 3.4.2.  By Lemma~\ref{CruxEntropyBound} the bracketing entropy integrals
\begin{eqnarray*} 
J_{[\,]} (\delta , {\cal G}^{(conv)} , L_2 (P_0)) \equiv \int_0^{\delta}
\sqrt { 1 + \log N_{[\, ]} (\epsilon , {\cal G}^{(conv)} , L_2 (P_0)) } \ d\epsilon 
\lesssim \int_0^\delta \epsilon^{-1/2} \left \{ \log (1/\epsilon) \right \}^{3\gamma_d/2} \ d \epsilon
\end{eqnarray*}
where the bound on the right side 
behaves asymptotically as a constant times $2 \delta^{1/2} ( \log (1/\delta ))^{3\gamma_d/2}$ 
with $3\gamma_d \equiv  5d/2 -2$,
and hence (using the notation of  Theorem 3.4.1 of \cite{MR1385671}), we can take 
$\phi_n (\delta ) = K 2 \delta^{1/2} ( \log (1/\delta ))^{3\gamma_d/2}$.     Thus with 
$r_n \equiv n^{1/3} / (\log n)^\beta $ with $\beta = \gamma_d$ we find that
$r_n^2 \phi_n (1/r_n ) \sim \tilde{K} \sqrt{n} $ and hence the claimed 
order of convergence holds.
\end{proof}

\section{Some related models and further problems}
\label{sec:OtherModelsFurtherProblems}

There are several related models in which we expect to see the same basic 
phenomenon as established here, namely a global convergence rate of the form
$n^{-1/3} (\log n)^{\gamma}$ in all dimensions $d\ge 2$ with only the power
$\gamma$ of the log term depending on $d$.  Three such models are:\\
(a) the ``in-out model'' for interval censoring in $\RR^d$;\\
(b) the ``case 2'' multivariate interval censoring models studied by \cite{MR2489672}; and\\
(c) the scale mixture of uniforms model for decreasing densities in $\RR^{+d}$.\\
Here we briefly sketch why we expect the same phenomenon to hold 
in these three cases, even though we do not yet know pointwise convergence rates in any of these cases.

\subsection{The ``in-out model'' for interval censoring in $\RR^d$}
\label{subsec:InOut}

The ``in-out model'' for interval censoring in $\RR^d$ was explored in the case $d=2$ by 
\cite{Song/PHD}.  In this model $\underline{Y} \sim F$ on $\RR^2$, $R$ is a random rectangle 
in $\RR^2$ independent of $\underline{Y}$ 
(say $[\underline{U},\underline{V}] = \{ \underline{x} = (x_1, x_2) \in \RR^2 : \ U_1 \le x_1 \le V_1,  \ U_2 \le x_2 \le V_2 \}$
where $\underline{U}$ and $\underline{V}$ are random vectors in $\RR^2$ 
with $\underline{U} \le \underline{V}$ coordinatewise).  
We observe only $ (1_R (\underline{Y}), R)$, and the goal is to estimate the unknown distribution function
$F$.  

\cite{Song/PHD}  (page 86) produced a local asymptotic minimax lower bound for estimation of 
$F$ at a fixed $\underline{t}_0 \in \RR^2$.  Under the assumption that $F$ has a positive density $f$
at $\underline{t}_0$,  \cite{Song/PHD} showed that any estimator of $F(\underline{t}_0)$ can have a local-minimax
convergence rate which is at best $n^{-1/3}$.   \cite{Groeneboom:12} has shown that this rate can
be achieved by estimators involving smoothing methods.   Based on the results 
for current status data in $\RR^d$ obtained in Theorem~\ref{NewMultHellingerRateThm} 
and the entropy results for the class of distribution functions on $\RR^d$,
we conjecture that the global Hellinger rate of convergence of the MLE 
$\hat{F}_n (\underline{t}_0)$ will be  $n^{-1/3} (\log n)^{\nu}$ for all $d \ge 2$ where $\nu = \nu_d$.

\subsection{``Case 2'' multivariate interval censoring models in $\RR^d$}
\label{subsec:MultVarCase2}

Recall that ``case 2'' interval censored data on $\RR$ is as follows:
suppose that $\underline{Y} \sim F_0$ on $\RR^+$, the pair of observation times
$(U,V)$ with $U \le V$  determines a random interval $(U,V]$, and we observe 
$\underline{X} = ( \underline{\Delta}, U, V) = (\Delta_1, \Delta_2, \Delta_3, U,V)$
where $\Delta_1 = 1\{ Y \le U \}$,  $\Delta_2 = 1\{U<Y \le V\}$, and $\Delta_3 = 1\{ V < Y \}$.
Nonparametric estimation of $F_0$ based on $\underline{X}_1, \ldots , \underline{X}_n )$ 
i.i.d. as $\underline{X}$ has been discussed by a number of authors, including 
\cite{MR1180321},  \cite{MR1714713}, and \cite{MR1600884}.
\cite{MR2489672}  studied generalizations of this model to $\RR^d$, and obtained 
rates of convergence of the MLE with respect to the Hellinger metric given by 
$n^{-(1+d)/(2(1+2d)} (\log n)^{d^2/(2(2d+1)}$ in the case most comparable to 
the multivariate interval censoring model studied here.    While this rate reduces when $d=1$ to the 
known rate $n^{-1/3} (\log n)^{1/6}$, it is slower than $n^{-1/3} (\log n)^{\nu}$ for some $\nu$
when $d>1$ due to the use of entropy bounds involving convex hulls (see \cite{MR2489672}, 
Proposition A.1, page 66) which are not necessarily sharp.  We expect that rates 
of the form $n^{-1/3} (\log n)^{\nu}$ with $\nu>0$ are possible in these models as well.

\subsection{Scale mixtures of uniform densities on $\RR^{+d}$}
\label{subsec:ScaleMixUniform}

\cite{MR2717518} 
and  \cite{MR2890434}    
studied the family of scale mixtures of uniform densities of the following form:
\begin{eqnarray}
f_G (\underline{x}) = \int_{\RR^{+d}} \frac{1}{\prod_{j=1}^d y_j }  1_{(0,\underline{y}]} (\underline{x}) dG(\underline{y})\equiv 
\int_{\RR^{+d}} \frac{1}{|\underline{y}|}  1_{(0,\underline{y}]} (\underline{x}) dG(\underline{y})
\label{ScaleMixtureOfUniforms}
\end{eqnarray}
for some distribution function $G$ on $(0,\infty)^{d}$.   
(Note that we have used the notation  $\prod_{j=1}^d y_j = | \underline{y} | $ for 
$\underline{y} = (y_1, \ldots , y_d) \in \RR^{+d}$.)
It is not difficult to see 
that such densities are decreasing in each coordinate and that they also satisfy
\begin{eqnarray*}
(\Delta_d f_G) (\underline{u},\underline{v}] = 
(-1)^d \int_{(\underline{u},\underline{v}]} |\underline{y}|^{-1} 1_{(\underline{y},\underline{v}]} d G(\underline{y}) \ge 0
\end{eqnarray*}
for all $\underline{u},\underline{v} \in \RR^{+d} $ with $\underline{u} \le \underline{v}$; here 
$\Delta_d$ denotes the $d-$dimensional
difference operator.    This is the same key property of 
distribution functions which results in (bracketing) entropies which depend on 
dimension only through a logarithmic term.   The difference here is that 
the density functions $f_G$ need not be bounded, and even if the true density 
$f_0 $ is in this class and satisfies $f_0 (\underline{0}) < \infty$, then we do not yet 
know the behavior of the MLE $\hat{f}_n $ at zero.  In fact we conjecture that:
(a)  If $f_0 (\underline{0}) < \infty$ and $f_0$ is a scale mixture of uniform densities on rectangles as in 
(\ref{ScaleMixtureOfUniforms}), then $\hat{f}_n (\underline{0}) = O_p ((\log n)^{\beta})$ for some 
$\beta = \beta_d>0$. (b) Under the same hypothesis as in (a) and the hypothesis that 
$f_0$ has support contained in a compact set, 
the MLE converges with respect to the Hellinger distance with a rate that is no worse than
$n^{-1/3} (\log n)^{\xi}$ where $\xi= \xi_d$.   Again \cite{MR2717518} and \cite{MR2890434} 
establish asymptotic minimax lower bounds for estimation of $f_0 (\underline{x}_0)$ proving that
no estimator can have a (local minimax) rate of convergence faster than $n^{-1/3}$ in all dimensions.
This is in sharp contrast to the class of block-decreasing densities on $\RR^{+d}$ studied
by \cite{MR2911847} and by  \cite{MR1994726}:    \cite{MR2911847} shows that the local 
asymptotic minimax rate for estimation of $f_0 (x_0)$ is no faster than $n^{-1/(d+2)}$, 
while \cite{MR1994726} show that there exist (histogram type) estimators $\tilde{f}_n $ 
which satisfy $E_{f_0} \| \tilde{f}_n - f_0 \|_1 = O (n^{-1/(d+2)})$.

\section{Appendix}
\label{sec:Appendix}

We begin by summarizing the results of \cite{Gao:12}.    
For a (probability) measure $\mu$ on $[0,1]^d$, let $F \equiv F_{\mu}$ denote the 
corresponding distribution function given by 
$$
F(\underline{x}) = F_{\mu} (\underline{x}) = \mu ( [0,\underline{x}] ) = \mu ( [0,x_1]\times \cdots \times [0,x_d] )
$$
for all $\underline{x} = (x_1, \ldots , x_d) \in [0,1]^d$.  
Let ${\cal F}_d$ denote the collection of all distribution functions on $[0,1]^d$; i.e. 
\begin{eqnarray*}
{\cal F}_d = \{ F : \ F \ \ \mbox{is a distribution function on} \ \ [0,1]^d \} .
\end{eqnarray*}
For example, if $\lambda_d$ denotes Lebesgue measure on $[0,1]^d$, then the corresponding
distribution function is
$F(\underline{x}) = F_{\lambda_d} (\underline{x}) = \prod_{j=1}^d x_j$.
\smallskip
 
\begin{thm} (Gao, 2012).
\label{GaosThm}
For $d \ge 2$ and $1 \le p < \infty$ 
\begin{eqnarray*}
\log N_{[\,]} (\epsilon , {\cal F}_d , L_p (\lambda_d) ) \lesssim \epsilon^{-1} \left ( \log (1/\epsilon) \right )^{2(d-1)} 
\end{eqnarray*}
for all $0 < \epsilon \le 1$.
\end{thm}

Our goal here is to use this result to control bracketing numbers for ${\cal F}_d$ with respect to 
two other measures $C_d$ and $R_{d,\sigma}$ defined as follows.  Let $C_d$ denote the finite 
measure on $[0,1]^d$ with density with respect to $\lambda_d$ given by
\begin{eqnarray*}
c_d (\underline{u}) = \frac{d!}{d^d}  \prod_{j=1}^d \frac{1}{u_j^{1-1/d}} \cdot 1 \left \{ \sum_{j=1}^d u_j^{1/d} > d-1 \right \} .
\end{eqnarray*}
For fixed $\sigma > 0$, let $R_{d,\sigma}$ denote the (probability) measure on $(0,1]^d$ with 
density with respect to $\lambda_d$ given by 
\begin{eqnarray*}
r_{d,\sigma} (\underline{t}) = \frac{d!}{(\log(1/\sigma))^d} \frac{1}{\prod_{j=1}^d t_j }  1\left \{\prod_{j=1}^d t_j > \sigma \right \} .
\end{eqnarray*}
 
\begin{cor}
\label{CorollariesOfGaosThm}
(a) \ For each $d \ge 2$ it follows that for $\epsilon \le \epsilon_0 (d)$
\begin{eqnarray*} 
\log N_{[\, ]} (2^{d/2}\epsilon , {\cal F}_d , L_2 (C_d) ) 
\lesssim \epsilon^{-1} \left ( \log (1/\epsilon) \right )^{2(d-1)} .
\end{eqnarray*}
(b) \ For each $d \ge 2$ and $\sigma \le \sigma_0 (d)$ it follows that for $\epsilon \le \epsilon_0 (d)/2$
\begin{eqnarray*} 
\log N_{[\, ]} (2^{d/2 +1}\epsilon , {\cal F}_d , L_2 (R_{d,\sigma}) ) 
\lesssim \epsilon^{-1} \left ( \log (1/\epsilon) \right )^{2(d-1)} .
\end{eqnarray*}
\end{cor}
 
\begin{proof}
We first prove (a).  
 We set $p\equiv p_d = 2 r_d \equiv 2r$ where $r \equiv r_d = 2d-1$ 
and $s =(d-1/2)/(d-1)$ satisfy $r^{-1} + s^{-1} =1$.  
Let $\{ [g_j, h_j ], j = 1, \ldots , m\}$ be a collection of $\epsilon-$brackets for ${\cal F}_d$
with respect to $L_p (\lambda_d)$.  (Thus for $d=2$, $r = 3$, $s=3/2$, and $p=6$, while for 
$d=4$, $r=7$, $s= (13/2)/3 = 13/6$, and $p= 14$.)  
By Theorem A.1 we know that $m \lesssim \epsilon^{-1} ( \log (1/\epsilon) )^{2(d-1)} $.
Now we bound the size of the brackets 
$[g_j, h_j]$ with respect to $C_d$.  
Using  H\"older's inequality with $1/r + 1/s =1$ as chosen above we find that 
\begin{eqnarray}
\int_{[0,1]^d} ( h_j - g_j )^2 c_d (u) du
& \le & \left ( \int_{[0,1]^d} | h_j - g_j |^{2r}  \underline{du} \right )^{1/r} 
                \cdot \left ( \int_{[0,1]^d} c_d (\underline{u})^s \underline{du} \right )^{1/s} \nonumber \\
&\le & (\epsilon^p )^{1/r} \cdot 2^{d/s} \le 2^{d} \epsilon^2 .
\label{L2EntropyWrtCd}
\end{eqnarray} 

Here are some details of the computation leading to (\ref{L2EntropyWrtCd}):
\begin{eqnarray*}
\int_{[0,1]^d} c_d (\underline{u})^s \underline{du}
& = & \int_{[0,1]^d} \left ( \frac{d!}{d^d} \right )^s \prod_{j=1}^d \frac{1}{u_j^{(d-1/2)/d}} 
           \cdot 1 \left \{ \sum_{j=1}^d u_j^{1/d} > d-1 \right \} \underline{du} \\
& = & \left ( \frac{d!}{d^d} \right )^s \cdot (2d)^d \int_{[0,1]^d} 1 \left \{ \sum_{j=1}^d x_j^{2} > d-1 \right \} \underline{dx}\\
& \le & \left ( \frac{d!}{d^d} \right )^s \cdot (2d)^d \cdot \int_{[0,1]^d}1 \left \{ \sum_{j=1}^d x_j > d-1 \right \} \underline{dx}\\
& \le & \left ( \frac{d!}{d^d} \right )^s \cdot (2d)^d \cdot \int_{[0,1]^d}1 \left \{ \sum_{j=1}^d t_j < 1 \right \} \underline{dt}\\
&=&2^d\left ( \frac{d!}{d^d} \right )^{s-1}
\le 2^d .
\end{eqnarray*}

To prove (b) we introduce monotone transformations $t_j (u_j)$ and their inverses $u_j (t_j)$ which 
relate $c_d$ and $r_{d,\sigma}$:  we set 
\begin{eqnarray*}
&& u_j (t_j) \equiv \left ( \frac{ \log (t_j / \sigma )}{\log (1/\sigma)} \right )^d, \\
&& t_j (u_j) \equiv \sigma \exp ( u_j^{1/d} \log (1/\sigma) )
\end{eqnarray*}
for $j = 1, \ldots , m$.  These all depend on $\sigma>0$, but this dependence is suppressed in 
the notation.    

For the same brackets $[g_j , h_j]$ used in the proof of (a), we define new brackets 
$[ \tilde{g}_j , \tilde{h}_j ]$ for $j = 1, \ldots , m$ by 
\begin{eqnarray*}
&& \tilde{g}_j (\underline{t}) \equiv \tilde{g}_{j, \sigma} (\underline{t}) = g_j ( u(\underline{t})) = g_j (u_1 (t_1), \ldots , u_d (t_d) ) ,\\
&& \tilde{h}_j (\underline{t}) \equiv \tilde{h}_{j, \sigma} (\underline{t}) = h_j ( u(\underline{t})) = h_j (u_1 (t_1), \ldots , u_d (t_d))) .
\end{eqnarray*}
Then it follows easily by direct calculation using
\begin{eqnarray*}
\prod_{j=1}^d t_j &= & \sigma^d \exp \left ( \log (1/\sigma) \sum_{j=1}^d u_j^{1/d} \right ), \\
\underline{d t} & = & \prod_{j=1}^d \left \{ \sigma \exp ( \log (1/\sigma) u_j^{1/d} ) 
                          \cdot d^{-1} u_j^{1/d-1} \cdot \log (1/\sigma) (d u_j) \right \} \\
     & = & \frac{\sigma^d (\log (1/\sigma))^d}{d^d} \prod_{j=1}^d t_j \cdot \prod_{j=1}^d u_j^{- (1-1/d)} \cdot \underline{du} \\
 \left \{ \prod_{j=1}^d t_j > \sigma \right \} 
      & = & \left \{ \exp \left ( \log (1/\sigma) \sum_{j=1}^d u_j^{1/d} \right ) > \sigma^{- (d-1)} \right \} \\
      & = & \left \{ \log (1/\sigma) \sum_{j=1}^d u_j^{1/d} > (d-1) \log (1/\sigma) \right \} \\
      & = & \left \{ \sum_{j=1}^d u_j^{1/d} > d-1 \right \} ,
\end{eqnarray*}
that 
\begin{eqnarray*}
\int_{[0,1]^d} (\tilde{h}_j (t) - \tilde{g}_j (t))^2 r_{d, \sigma} (t) \underline{dt}
& = & \int_{[0,1]^d} ( h_j (u) - g_j (u))^2 c_d (u) \underline{du}  .
\end{eqnarray*} 
Thus for $\sigma\le \sigma_0 (d)$ we have
\begin{eqnarray*}
\| \tilde{h}_j - \tilde{g}_j \|_{L_2 (R_{d, \sigma })} \le 2^{d/2+1} \epsilon
\end{eqnarray*}
by the arguments in (a). 
Hence the brackets 
$[\tilde{g}_j , \tilde{h}_j]$ yield a collection of $2^{d/2+1} \epsilon -$ brackets for ${\cal F}_d$ 
with respect to $L_2 (R_{d,\sigma)}$, and this 
implies that (b) holds.
\end{proof}
\bigskip

\par\noindent
{\bf Acknowledgements:}  We owe thanks to the referees for a number of helpful suggestions
and for pointing out the work of \cite{MR2236498} and \cite{MR2489672}.


\end{document}